\documentclass[12pt]{article}

\usepackage{fixdif, mathtools, keytheorems, amssymb, physics2, enumitem, microtype, authblk, zref-clever}
\usephysicsmodule{ab, ab.braket}
\allowdisplaybreaks
\usepackage[hidelinks, colorlinks, citecolor = red]{hyperref}

\theoremstyle{plain}
\newkeytheorem{theorem}[parent = section]

\newkeytheorem{lemma}[parent = section]

\newtheorem{proposition}{Proposition}[section]

\numberwithin{equation}{section}

\begin{document}
	
\title {Normalized Solutions for Schr\"odinger-Bopp-Podolsky Systems
	with Critical Choquard-Type Nonlinearity on Bounded Domains}
\author{Li Chen,\quad
	Li Wang\footnote{\textit{Corresponding author.}
	chenlimath0318@163.com (L. Chen), wangli.423@163.com (L. Wang).}}
\affil{College of Science, East China Jiaotong University, Nanchang, 330013, China}
\maketitle
	
	\begin{abstract}
In this paper, we study normalized solutions for the following critical Schr\"odinger-Bopp-Podolsky system:
\[
	\begin{cases}
	-\Delta u + q(x)\phi u = \lambda u + |u|^{p-2}u + \bigl(I_\alpha * |u|^{3+\alpha}\bigr)|u|^{1+\alpha}u, & \text{in } \Omega_r, \\
	-\Delta\phi + \Delta^2\phi = q(x)u^2, & \text{in } \Omega_r,
	\end{cases}
\]
where \(\Omega_r \subset \mathbb R^3\) is a smooth bounded domain, \(p \in \left(2, \frac{8}{3}\right)\), $q(x) \in C(\bar\Omega_r) \backslash \{0\}$ and \(\lambda \in \mathbb R\) is the Lagrange multiplier associated with the constraint \(\int_{\Omega_r} |u|^2 \d x = b^2\) for some \(b > 0\). Here \(\alpha > 0\), \(I_\alpha\) denotes the Riesz potential, and the domain parameter \(r\) reflects the size of \(\Omega_r\) whose precise definition will be given in Section 3.
By applying a special minimax principle together with a truncation technique, we prove that there exists \(b^* > 0\) such that the system admits multiple normalized solutions whenever \(b \in (0, b^*)\) under Navier boundary conditions.
	\end{abstract}

	{\footnotesize {\bfseries Keywords: }	Schr\"odinger-Bopp-Podolsky system, Normalized solutions, Critical Choquard nonlinearity.
		
	}

	\section{Introduction and main result}

\qquad
Independently developed by Bopp~\cite{1} and Podolsky~\cite{2}, the Bopp-Podolsky theory constitutes a second-order gauge field theory for electromagnetism. It was originally proposed to resolve the well-known infinity problem inherent in the classical Maxwell model. For subsequent developments and related results within this framework, we refer to~\cite{3,4,5,6,7,8}.

Within the classical electrostatic formulation, the potential $\phi$ generated by a charge density $\rho$ satisfies the Poisson equation
\[
-\Delta\phi=\rho \quad \text{in } \mathbb R^3.
\]
In the case of a point charge located at \(x_{0}\in\mathbb R^3\), i.e., \(\rho=4\pi\delta_{x_{0}}\), the solution becomes \(\phi(x)=|x-x_{0}|^{-1}\). Consequently, the associated electrostatic energy diverges:
\[
\int_{\mathbb R^3}|\nabla\phi|^2\,\d x=+\infty.
\]
To remedy this divergence, Bopp and Podolsky introduced a modified field equation
\[
-\Delta\phi+a^2\Delta^2\phi=\rho \quad \text{in } \mathbb R^3,
\]
where \(a>0\) is the Bopp-Podolsky parameter. For the same point charge \(\rho=4\pi\delta_{x_{0}}\), the explicit solution of the modified equation is
\[
\mathcal{K}(x)=\frac{1-e^{-|x|/a}}{|x|}, \qquad x\in\mathbb R^3.
\]
This potential remains finite at the origin and satisfies the finiteness condition
\[
\int_{\mathbb R^3}|\nabla\mathcal{K}|^2\,\d x+a^2\int_{\mathbb R^3}|\Delta\mathcal{K}|^2\,\d x<+\infty.
\]
Thus, within the Bopp-Podolsky framework, the electrostatic field produced by a point charge possesses finite energy when measured by the BP norm \(\int_{\mathbb R^3}\big(|\nabla\phi|^2+a^2|\Delta\phi|^2\big)\d x\).

From an electrostatic perspective, coupling the Bopp-Podolsky equation with a nonlinear Schr\"odinger equation is a natural step~\cite{9,16}. The resulting coupled systems, referred to as Schr\"odinger-Bopp-Podolsky systems, can be considered as generalizations of the well-known Schr\"odinger-Poisson or Schr\"odinger-Maxwell frameworks.

In recent years, coupled systems involving nonlinear Schr\"odinger equations and various electromagnetic field theories have attracted considerable research interest. Among them, the Schr\"odinger-Bopp-Podolsky system, serving as a generalization of the classical Schr\"odinger-Poisson model, has been a subject of significant attention in mathematical physics. This system arises from coupling the Schr\"odinger equation with the Bopp-Podolsky electromagnetic theory, a second-order gauge theory proposed to resolve the energy divergence of point charges inherent in the classical Maxwell theory. In the electrostatic framework, the system can be formulated as:
\begin{equation}
	\begin{cases}
		-\Delta u + \omega u + q^2\phi u = |u|^{p-2}u, & \text{in } \mathbb R^3,\\
		-\Delta\phi + a^2\Delta^2\phi = 4\pi u^2, & \text{in } \mathbb R^3,
	\end{cases}
	\label{eq:1.1*}
\end{equation}
where \(a > 0\) is the Bopp-Podolsky parameter, \(\omega > 0\), \(q \neq 0\), and the nonlinear exponent satisfies \(p \in (2,6)\).

In paper~\cite{9}, d'Avenia and Siciliano presented the first systematic study of standing waves for the Schr\"odinger-Bopp-Podolsky system~\eqref{eq:1.1*} in $\mathbb R^3$.
 By variational methods, the authors constructed the corresponding energy functional and addressed the lack of compactness caused by nonlinear growth and translation invariance through techniques such as truncation, the monotonicity trick, and topological tools.

In recent years, Schr\"odinger-type equations with nonlocal terms have attracted considerable attention in mathematical physics. Among these, the Schr\"odinger-Bopp-Podolsky system is of particular interest due to its physical background in generalized electrodynamics. In the existing literature, Figueiredo and Siciliano in~\cite{10} conducted a thorough investigation on the multiplicity of solutions for this system under positive potential \( V \). 
The system describes the electrostatic interaction of a charged particle within the Bopp-Podolsky framework and takes the form:
\[
\begin{cases}
-\varepsilon^2 \Delta u + V u + \lambda \phi u + f(u) = 0, \\
-\varepsilon^2 \Delta \phi + \varepsilon^4 \Delta^2 \phi = u^2.
\end{cases}
\]
The authors employed the Krasnoselski genus theory to prove that the system admits infinitely many solutions whose energy and norm both tend to infinity, and the associated Lagrange multipliers tend to \( -\infty \).
Moreover, under some appropriate assumptions of $f$ and $V$, they obtained a
ground state solution via minimization, which can be chosen as a negative solution.
They provided a robust framework for studying multi-solution phenomena in nonlocal Schr\"odinger-type systems and inspires further investigation into problems with sign-changing or degenerate nonlinearities.

Recently, Li, Pucci and Tang in~\cite{15} investigated the following system with critical Sobolev exponent:
\[
\begin{cases}
-\Delta u + V(x)u + q^2 \phi u = \mu |u|^{p-1}u + |u|^4 u & \text{in } \mathbb R^3, \\
-\Delta \phi + a^2 \Delta^2 \phi = 4\pi u^2 & \text{in } \mathbb R^3,
\end{cases}
\]
where \(\mu > 0\) and \(2 < p < 5\).
Taking some technical assumptions to the potential $V$,
they established the existence of a nontrivial ground state solution. Their approach relied on the Poho\v zaev-Nehari manifold method, combined with Br\'ezis-Nirenberg-type arguments, the monotonicity trick, and a global compactness lemma. They proved that for \(p \in (3,5)\), the system admits a ground state for any \(\mu > 0\), while for \(p \in (2,3]\), a ground state exists provided \(\mu\) is sufficiently large. This work extends earlier results to the critical case and provides a foundation for further studies on nonlinearly coupled field systems with lack of compactness and translation invariance.

In studies on the Schr\"odinger-Bopp-Podolsky system, besides the critical exponent case, Silva and Siciliano in~\cite{18} investigated the following system:
\[
\begin{cases}
-\Delta u + \omega u + q^2 \phi u = |u|^{p-2}u, & \text{in }\mathbb R^3,\\
-\Delta \phi + a^2 \Delta^2 \phi = 4\pi u^2, & \text{in }\mathbb R^3,
\end{cases}
\]
where \( p \in (2,3] \), \( a, \omega > 0 \), and \( q \neq 0 \). By employing fibering map analysis, the Nehari manifold method, and the Mountain Pass Theorem, they proved the existence of two critical parameter values \( q^* \) and \( q_0^* \): when \( q > q^* \), the system admits no nontrivial solution; whereas for \( q < q_0^* \), the system possesses two positive solutions in the radial function space \( H_r^1(\mathbb R^3) \), one being a global or local minimizer and the other a Mountain Pass type solution. This result reveals the crucial role of the charge parameter \( q \) on the structure and existence of solutions in the subcritical growth range.
For more research paper on the Schr\"odinger-Bopp-Podolsky system in the whole
space, interested authors may refer to~\cite{11,14,16}.

The study of Schr\"odinger-Bopp-Podolsky systems has also attracted considerable attention on bounded domains. Afonso and Siciliano in~\cite{19}
studied the following system:
\begin{equation}
	\begin{cases} 
-\Delta u + q\phi u - \kappa |u|^{p-2}u = \omega u, & \text{in } \Omega, \\ 
- \Delta \phi + \Delta^2 \phi = qu^2, & \text{in } \Omega.
	\end{cases}
\end{equation}
The system devotes under Neumann boundary conditions \(\frac{\partial \phi}{\partial \mathbf{n}} = h_1,\ \frac{\partial \Delta \phi}{\partial \mathbf{n}} = h_2\) and subject to the normalization constraint \(\int_{\Omega} u^2 \d x = 1\). Under the assumptions \(\kappa > 0,\; p \in (2, \frac{10}3)\) and provided that the boundary integral combination \(\alpha = \int_{\partial \Omega} h_2 \d s - \int_{\partial \Omega} h_1 \d s\) satisfies \(\inf_{\Omega} q < \alpha < \sup_{\Omega} q\) together with \(|q^{-1}(\alpha)| = 0\), they proved, via a combination of variational methods and topological theory, that the problem admits infinitely many solutions \((u_n, \omega_n, \phi_n)\) with \(\|\nabla u_n\|_2 \to +\infty\).

Based on the above work, we consider the following Schr\"odinger-Bopp-Podolsky system on a smooth bounded domain \(\Omega_r \subset \mathbb R^3\):
\begin{equation}
	\begin{cases}
-\Delta u + q(x) \phi u = \omega u + |u|^{p-2}u + (I_\alpha * |u|^{3+\alpha})|u|^{1+\alpha}u, & \text{in } \Omega_r, \\[4pt]
-\Delta \phi + \Delta^2 \phi = q(x) u^2, & \text{in } \Omega_r, \\[4pt]
u = 0, & \text{on } \partial\Omega_r, \\[4pt]
\int_{\Omega_r} u^2 \d x = b^2, & \text{in } \Omega_r.
\end{cases}
\label{eq:1.1}
\end{equation}
In this system, \(\lambda \in \mathbb R\) is an unknown frequency arising as a Lagrange multiplier, \(b > 0\) is a prescribed mass, and the function \(q \in C(\overline{\Omega}) \setminus \{0\}\) represents a nonuniform charge distribution.
For the electrostatic potential $\phi$,
we impose homogeneous Navier-type boundary conditions:
\begin{equation}
	\phi = 0, \quad \text{on } \partial\Omega_r.
	\label{eq:1.2}
\end{equation}
\begin{equation}
	\Delta\phi = 0 \quad \text{on } \partial\Omega_r.
	\label{eq:1.3}
\end{equation}
The choice of zero boundary data is made here for technical simplicity.

The primary research finding of this paper is presented as follows:

	\begin{theorem}[label=thm:1.1]
	For any given $k \in \mathbb N$, there exist a constant $b^* > 0$, such that for every $b \in (0,b^*)$, the equation~\eqref{eq:1.1} subject to the boundary conditions~\eqref{eq:1.2}-\eqref{eq:1.3} admits at least k solutions
	\[
	(u_j, \phi_j, \lambda_j) \in H_0^1(\Omega_r) \times H^2(\Omega_r) \times \mathbb R,
	\]
	satisfying the constraint
	\[
	\int_{\Omega_r} u_j^2 \d x = b^2
	\]
	for any $j = 1, 2, \cdots, k$.
\end{theorem}

The remainder of this paper is structured as follows. In Section 2, we establish the variational framework for problem~\eqref{eq:1.1}, reduce the natural energy functional to a single-variable functional depending solely on \(u\), and derive refined upper bounds for the norm of the electrostatic potential. Section 3 is devoted to discussing the multiplicity of normalized solutions and provides the proof of~\zcref[S]{thm:1.1}.

	\section{Preliminaries}

\qquad

Throughout this paper, let \(\Omega \subset \mathbb R^3\) be a fixed bounded smooth domain. For any \(r > 0\), we define the scaled domain
\[
\Omega_{r} = \{ r x \in \mathbb R^3 : x \in \Omega \}.
\]
The precise value of the scaling parameter \(r\) will be related to another parameter \(t_0\) and determined later in Section 3.

We denote by $H^2(\Omega_{r})$ and $H_{0}^{1}(\Omega_{r})$ the usual Sobolev spaces. In particular, the inner product and norm on $H_{0}^{1}(\Omega_{r})$ are defined respectively as
	$$(u,v):=\int_{\Omega_{r}}\nabla u\cdot \nabla v\d  x, \quad \|u\|:=\left(\int_{\Omega_{r}}|\nabla u|^2\d x\right)^{\frac{1}{2}}.$$
For a prescribed mass $b > 0$, we define $L^2$-constrained sphere
	\[
	 S_{r,b} := \{u \in H_0^1(\Omega_r): \|u\|_2^2 = b,\ b > 0\}.
	\]
	For $1\leq p\leq 6$, we use the notation
	\[
		\|u\|_{p}:=\left(\int_{\Omega_{r}}|\nabla u|\d x\right)^{\frac{1}{p}}
	\]
	to denote the norm in $L^{p}(\Omega_{r})$, and $\|u\|_{\infty}$ stands for the norm in $L^{\infty}(\Omega_{r})$. For the sake of brevity, we omit the spatial variable $x\in\Omega_{r}$ in all functions in what follows, unless specified otherwise to emphasize that the function is non-constant.

Within the variational framework, the boundary value problem consisting of system~\eqref{eq:1.1} together with conditions~\eqref{eq:1.2}-\eqref{eq:1.3} corresponds to the functional
\begin{align*}
	\mathcal F(u, \phi) & = \frac12 \int_{\Omega_r} |\nabla u|^2 \d x
	- \frac1{2 (3+\alpha)} \int_{\Omega_r} (I_a * |u|^{3+\alpha})|u|^{3 + \alpha}\d x
	+ \frac12 \int_{\Omega_r} q(x) \phi u^2 \d x \\
	&\quad - \dfrac{1}{p}\int_{\Omega_r} |u|^{p} \d x
	- \frac14 \int_{\Omega_r} |\Delta\phi|^2 \d x - \frac14 \int_{\Omega_r} |\nabla\phi|^2 \d x
\end{align*}
defined on the product space $H_0^1(\Omega_r) \times H^2(\Omega_r)$.
This functional, however, fails to be bounded either from above or from below;
consequently, standard critical-point techniques cannot be applied directly.
Following the reduction strategy employed in~\cite{benci1998eigenvalue,20}, we first resolve the second equation in~\eqref{eq:1.1} to express the potential uniquely as $\phi = \Phi(u)$. Substituting this relation yields a reduced functional that depends solely on the variable $u$.
	
Recall that under the homogenous Navier conditions $\phi = \Delta\phi = 0$ on $\partial\Omega_r$, if $\phi=\Delta\phi=0$ holds on $\Omega_{r}$, the following elliptic estimate hold: there exists a constant $C=C(\Omega_{r})>0$ such that
	$$\int_{\Omega_{r}}|D^2\phi|^2\d  x\leq C\int_{\Omega_{r}}|\nabla\phi|^2\d x, \quad \forall \phi\in H_{0}^{1}(\Omega_{r})\cap H^2(\Omega_{r}),$$
	we denote
	$$\mathbb{H}:=H_{0}^{1}(\Omega_{r})\cap H^2(\Omega_{r}), \ \ \ \ \ \ \ \|\phi\|_{\mathbb{H}}:=\left(\int_{\Omega_{r}}|\nabla\phi|^2\d  x+\int_{\Omega_{r}}|\Delta\phi|^2\d  x\right)^{\frac{1}{2}}$$
	with inner product
	$$(\phi,\psi)_{\mathbb{H}}:=\int_{\Omega_{r}}\nabla\phi\nabla\psi\d  x+\int_{\Omega_{r}}\Delta\phi\Delta\psi\d  x.$$
	
	As mentioned earlier, the $\lambda$ appears as the Lagrange multiplier associated with the mass constraint $\int_{\Omega_r} u^2\d x = b^2$. The set $S_{r,b}=\left\{u\in H_{0}^{1}(\Omega_{r}):\|u\|_2^2=b^2\right\}$ will be the natural constraint set for our variational problem.
	
	The following section is developed within the framework of the Navier boundary conditions~\eqref{eq:1.2}-\eqref{eq:1.3}. A trip $(u,\phi,\lambda)\in H_{0}^{1}(\Omega_{r})\times\mathbb{H}\times\mathbb R$ is termed a weak solution of system~\eqref{eq:1.1} under the boundary conditions~\eqref{eq:1.2}-\eqref{eq:1.3} if it satisfies the following integral identities:
	\begin{align}
		\begin{split}
			&\int_{\Omega_{r}}\nabla u\cdot\nabla\varphi\d  x+\int_{\Omega_{r}}q(x)\phi u\varphi\d  x-\int_{\Omega_{r}}|u|^{p-2}u\varphi\d  x - \lambda \int_{\Omega_r} u\varphi \d x\\
			-{}&\int_{\Omega_{r}}(I_{\alpha}*|u|^{3+\alpha})|u|^{1+\alpha}u\varphi\d  x=0, \ \ \ \ \forall\varphi\in H_{0}^{1}(\Omega_{r})
		\end{split}
		\label{eq:2.1}
	\end{align}
	and
	\begin{align}
		\int_{\Omega_{r}}\nabla\phi\nabla\varphi\d  x + \int_{\Omega_{r}}\Delta\phi\Delta\varphi\d  x =\int_{\Omega_{r}}q(x)u^2\varphi\d  x, \ \ \ \ \forall\varphi\in\mathbb{H}.
	\end{align}
	
	We consider the natural energy functional $\mathcal{F}:H_{0}^{1}(\Omega_{r})\times\mathbb{H}\to\mathbb R$, derived as
	\begin{align}
		\begin{split}
			\mathcal{F}(u,\phi)={}&\dfrac{1}{2}\int_{\Omega_{r}}|\nabla u|^2\d  x+\dfrac{1}{2}\int_{\Omega_{r}}q(x)\phi u^2\d  x-\dfrac{1}{2(3+\alpha)}\int_{\Omega_{r}}(I_{\alpha}*|u|^{3+\alpha})|u|^{3+\alpha}\d  x\\
			&-\dfrac{1}{p}\int_{\Omega_{r}}|u|^{p}\d x-\dfrac{1}{4}\int_{\Omega_{r}}|\Delta\phi|^2\d x-\dfrac{1}{4}\int_{\Omega_{r}}|\nabla\phi|^2\d x.
		\end{split}
		\label{eq:2.3}
	\end{align}
	It can be directly verified that $\mathcal{F}\in C^{1}(H_{0}^{1}(\Omega_{r})\times\mathbb{H},\mathbb R)$, with partial derivatives given by:
	\begin{align}
		\begin{split}
			\braket<\partial_{u}	\mathcal{F}(u,\phi),\varphi>&=\int_{\Omega_{r}}\nabla u\nabla \varphi\d  x+\int_{\Omega}q(x)\phi u\varphi\d  x-\int_{\Omega_{r}}|u|^{p-2}u\varphi\d  x\\
			&\quad -\int_{\Omega_{r}}(I_{\alpha}*|u|^{3+\alpha})|u|^{1+\alpha}u\varphi\d  x, \ \ \ \ \forall\varphi\in H_{0}^{1}(\Omega_{r}),
		\end{split}\\
		\begin{split}
			\braket<\partial_{\phi}	\mathcal{F}(u,\phi), \xi>&=\dfrac{1}{2}\int_{\Omega_{r}}q(x)u^2\xi\d  x-\dfrac{1}{2}\int_{\Omega_{r}}\Delta\phi\Delta\xi\d  x-\dfrac{1}{2}\int_{\Omega_{r}}\nabla\phi\nabla\xi\d  x, \ \ \ \ \forall\xi\in \mathbb{H}.
		\end{split}
	\end{align}
	\begin{lemma}
		Let $(u,\phi,\lambda)\in H_{0}^{1}(\Omega_{r})\times\mathbb{H}\times\mathbb R$. Then $(u,\phi,\lambda)$ is a weak solution of~\eqref{eq:1.1} under
		the boundary condition \eqref{eq:1.2}-\eqref{eq:1.3} if and only if $(u, \phi)$ is a critical
		point of the functional $\mathcal F$ defined in~\eqref{eq:2.3}, subject to the
		constraint $S_{r,b} \times \mathbb H$. In this setting, $\lambda$
		represents the Lagrange multiplier associated with the constraint.
	\end{lemma}
		Consequently, any critical point of $\mathcal F$ restricted to the set
		\[
			S_{r,b} \times \mathbb H =
			\{
				(u,\phi) \in H_0^1(\Omega_r) \times \mathbb H:\ 
				\|u\|_2^2 = b^2
			\}
		\]
		yields a weak solution of system~\eqref{eq:1.1}
		under~\eqref{eq:1.2}-\eqref{eq:1.3}.
		The real number $\lambda$ appearing in~\eqref{eq:2.1}
		then corresponds to the
		associated Lagrange multiplier. The proof of this result follows a line of
		reasoning entirely analogous to that presented in \cite{20},
		and is therefore omitted here.

		Consider a fixed element $u \in H_0^1(\Omega_r)$. We introduce the linear
		functional
		\[
			L_u: \mathbb H \to \mathbb R,\quad
			L_u(\xi) := \int_{\Omega_r} q(x) u^2\xi \d x.
		\]
		Applying H\"older's inequality along with the Sobolev
		embedding, we deduce for
		any $\xi \in \mathbb H$ that
		\begin{equation}
			|L_u(\xi)| \leq \|q(x)\|_\infty \|u\|_4^2 \|\xi\|_2 \leq C_1\|q(x)\|_\infty
			\|u\|_4^2 \|\xi\|_{\mathbb H} \leq C_2\|\nabla u\|_2^2 \|\xi\|_{\mathbb H},
			\label{eq:2.6*}
		\end{equation}
		where the constant $C_1$, $C_2 > 0$ depend only on  $\Omega_r$.
		This confirms the continuity of $L_u$ on $\mathbb H$.
		According to the Riesz representation theorem, there exists a unique element
		$\Phi(u) \in \mathbb H$ satisfying
		\begin{equation}
			L_u(\xi) = (\Phi(u),\xi)_{\mathbb H}
		= \int_{\Omega_r} \nabla\Phi(u) \cdot \nabla\xi \d x
		+ \int_{\Omega_r} \Delta\Phi(u) \Delta\xi \d x, \quad
		\forall \xi \in \mathbb H.
		\label{eq:2.7*}
		\end{equation}
		
		Equivalently, for each $u \in H_0^1(\Omega_r)$, the Bopp-Podolsky
		equation
		\[
			-\Delta\phi + \Delta^2\phi = q(x)u^2, \quad \text{in $\Omega_r$},\quad
			\phi = \Delta\phi = 0 \quad \text{on $\partial\Omega_r$}
		\]
	\noindent possesses a unique weak solution given
	by $\phi = \Phi(u) \in \mathbb H$.

	Substituting $\xi = \Phi(u)$ in~\eqref{eq:2.7*} yields the relation
	\begin{equation}
		(\Phi(u), \Phi(u))_{\mathbb H} = \int_{\Omega_r} q(x) u^2\Phi(u) \d x,
		\label{eq:2.6}
	\end{equation}
	which will be frequently employed in subsequent discussions.
	Utilizing~\eqref{eq:2.3} and \eqref{eq:2.6}, we can obtain
	\begin{align*}
		\mathcal F(u,\Phi(u)) & = \frac12 \int_{\Omega_r} |\nabla u|^2 \d x
		+ \frac12 \int_{\Omega_r} q(x) u^2 \Phi(u) \d x
		- \frac{1}{2(3 + \alpha)} \int_{\Omega_r} (I_\alpha * |u|^{3+\alpha})
			|u|^{3+\alpha} \d x\\
		& \quad - \dfrac{1}{p}\int_{\Omega_r} |u|^{p} \d x - \frac14\|\Phi(u)\|^2_{\mathbb H}\\
	& = \frac12 \int_{\Omega_r} |\nabla u|^2 \d x
	- \dfrac{1}{p}\int_{\Omega_r} |u|^{p} \d x + \frac14 \int_{\Omega_r} q(x) u^2 \Phi(u) \d x\\
	& \quad- \frac1{2(3 + \alpha)} \int_{\Omega_r} (I_\alpha * |u|^{3+\alpha})
		|u|^{3+\alpha} \d x.
	\end{align*}
	
	Based on this, we define the reduced functional
	\[
		J(u) := \frac12 \int_{\Omega_r} |\nabla u|^2\d x
	- \dfrac{1}{p}\int_{\Omega_r} |u|^{p} \d x + \frac14 \int_{\Omega_r} q(x) u^2 \Phi(x) \d x
	- \frac1{2(3 + \alpha)} \int_{\Omega_r} (I_\alpha * |u|^{3+\alpha})
		|u|^{3+\alpha} \d x.
	\]
	Similar to the discussion in~\cite{201}, assume $b > 0$
	and $u \in S_{r,b}$. Then $u$ is the critical point of $J$ on $S_{r,b}$
	if and only if $(u, \lambda, \Phi(u))$ is the weak solution to~\eqref{eq:1.1} under the
	conditions~\eqref{eq:1.2}-\eqref{eq:1.3}, where $\lambda \in \mathbb R$ is the Lagrange
	multiplier associated with the constraint.

	Then, we introduce some useful inequalities.
	\begin{lemma}[label=lemma:2.2*, note={Gagliardo-Nirenberg inequality, \cite{22}}]
		Let $2 < s < 2^* = \frac{2N}{N - 2}$, then there exists a sharp constant $C_{N,s}$ such that for all $u \in H_0(\Omega_r)$
		\[
			\|u\|_s^s \leq C_{N,s}\|u\|_2^{\frac{2s - N(s - 2)}{2}} \|\nabla u\|_2^{\frac{N(s - 2)}{2}}.
		\]
	\end{lemma}
		Then there exists a sharp constant $C(N, \lambda, r, r)$ independent of $f$,
		$g$, such that (Hardy Littlewood-Sobolev inequality,~\cite{23})
		\begin{equation}
			\ab|\int_{\mathbb R^N} \int_{\mathbb R^N}
			\frac{f(x)g(y)}{|x - y|^\lambda} \d x \d y| \leq C(N, \lambda, r, s)
			\|f\|_{L^r(\mathbb R^N)} \|g\|_{L^s(\mathbb R^N)},
			\label{eq:2.9*}
		\end{equation}
		Let $r$, $s > 1$, $0 < \lambda < N$ with $\frac1r + \frac1s + \frac\lambda N = 2$, $f \in L^r(\mathbb R^N)$, $g \in L^s(\mathbb R^N)$.
		
		From Hardy-Littlewood-Sobolev inequality, we can define the best constant
		\[
			S_{h,l} := \inf_{u\in D^{1,2}(\mathbb R^N)\backslash\{0\}}
			\frac{\int_{\mathbb R^N}|\nabla u|^2\d x}
				{[\int_{\mathbb R^N}(I_\alpha * |u|^{3+\alpha})|u|^{3+\alpha} \d x]^{\frac{1}{3+\alpha}}}.
		\]

		Next, we review the definitions related to the genus and state theminimax
		theorem that will be used in the subsequent sections.

		Let $X$ be a Banach space and $A \subseteq X$; the set $A$ is called
		symmetric if $u \in A$ entails $-u \in A$. Denote: $\Sigma:= \{A \subset X\backslash\{0\}: \text{ A is closed and symmetric with respect to the origin}\}$.

		For $A \in \Sigma$, set
		\[
			\gamma(A) :=
			\begin{cases}
				0, & \text{if $A = \emptyset$},\\
				\inf\{k \in \mathbb N: \exists \text{ and an odd $\varphi \in C(A, \mathbb R^k\backslash\{0\})$}\},\\
				+\infty, & \text{If there is no such odd map exists}.
			\end{cases}
		\]
		and $\Sigma_k := \{A \in \Sigma: \gamma(A) \geq k\}$.

		Let us now define the following manifold
		\[
			S_{r,b} := \{u \in H_0^1(\Omega_r) | (u,u)_{L^2} = b^2\},
		\]
		which is endowed with the topology inherited from $H_0^1(\Omega_r)$.
		The tangent space of $S_{r,b}$ at a point $u \in S_{r,b}$ is defined by
		\[
			T_uS_{r,b} = \{v \in H_0^1(\Omega_r): (u,v)_{L^2} = 0\}.
		\]
		Let $J \in C^1(H_0^1(\Omega_r), \mathbb R)$, then $J\big|_{S_{r,b}}$
		belongs to $C^1$. The norm of derivative of $J\big|_{S_{r,b}}$ at any point
		$u \in S_{r,b}$ is given by
		\[
			\|J\big|_{S_{r,b}}(u)\| = \sup_{\|v\| \leq 1,\ v \in T_u S_{r,b}} |\braket<J'(u), v>|.
		\]
		We note that $S_{r,b}$ is symmetric with respect to $0 \in S_{r,b}$ and
		$0 \notin S_{r,b}$. Let $\Sigma(S_{r,b})$ be the family of closed symmetric subsets of $S_{r,b}$, for each $j \in S_{r,b}$, let $\Gamma_j = \{A\in \Sigma(S_{r,b})|\gamma(A) \geq j\}$.
	\begin{proposition}\cite{peng2024normalized}
		Let $J \in C^1 (H_0^1(\Omega_r), \mathbb R)$ be an even functional.
		Assume that $\Gamma_j \neq \emptyset$ for each $j \in \mathbb N$,
		$J\big|_{S_{r,b}}$ is bounded from below and satisfies the $(PS)_d$
		condition for all $d < 0$. Define
		\[
			d_j: = \inf_{A\in \Gamma_j} \sup_{u\in A} J(u), \quad
			j = 1,2,\ldots,n.
		\]
		Then, the following statements hold:
		\begin{enumerate}[label = (\roman*)]
			\item $-\infty < d_1 \leq d_2 \leq \cdots \leq d_n$ and $d_j$ ($j = 1,2,\ldots,n$) is a critical value of $J\big|_{S_{r,b}}$ if $d_j < 0$.
			\item If $d: = d_j = d_{j+1} = \cdots = d_{j+l-1} < 0$ for some $j$,
			$l \geq 1$, then $\gamma(K_d) \geq l$, where $K_d$ denotes the set of
			critical points of $J\big|_{S_{r,b}}$ at the level $d$.
			In particular, $J\big|_{S_{r,b}}$ admits at least $n$ critical points
			at the level $d$.
		\end{enumerate}
	\end{proposition}

		Now, we introduce some properties of $\Phi(u)$, which can be found in~\cite{20}.
	\begin{lemma}[label=lemma:2.3]\leavevmode
		\begin{enumerate}[label = (\roman*)]
			\item If $u_n \rightharpoonup u$ in $H_0^1(\Omega_r)$, then
			\[
				\int_{\Omega_r} q(x) u_n^2 \Phi(u_n) \d x \to \int_{\Omega_r} q(x)
				u^2\Phi(u) \d x.
			\]
			Moreover, the map $\Phi$ is compact.
			\item If $u_n \rightharpoonup u$ in $H_0^1(\Omega_r)$, then
			\[
				\int_{\Omega_r}q(x)\Phi(u_n)u_n\varphi\d x
				\to \int_{\Omega_r} q(x) \Phi(u)u\varphi \d x.
			\]
		\end{enumerate}
	\end{lemma}

	At the end of this section, we revisit the concentration-compactness principle which will be used to overcome the lack of compactness of the functional.
	\begin{lemma}[label=lemma:2.4, note=\cite{24}]
		Let $\{u_n\}$ be a bounded sequence in $\mathcal D^{1,2}(\mathbb R^3)$
		converging weakly and almost everywhere to some $u \in \mathcal D^{1,2}(\mathbb R^3)$. Let
		\[
			|\nabla u_n|^2 \rightharpoonup \mu,
		\]
		\[
			(I_\alpha * |u_n|^{3+\alpha}) |u_n|^{3+\alpha} \rightharpoonup \nu
		\]
		weakly in the measure space, where $\mu$, $\nu$ are bounded positive measures on $\mathbb R^3$.
		Then there exists a set $I$ (at most countable) and
		$\{\mu_j\}_{j\in I}$, $\{\nu_j\}_{j\in I}\subset [0, \infty)$ such that
		\[
			\mu \geq |\nabla u|^2 + \sum_{j\in I} \mu_i \delta_{z_j},
		\]
		\[
			\nu = (I_\alpha * |u|^{3+\alpha})|u|^{3+\alpha} + \sum_{j\in I}
			\nu_j \delta_{z_j}, \quad
			\sum_{j\in I}\nu_j^{\frac12} < \infty,
		\]
		and
		\begin{equation}
			\mu_j \geq S_{h,l} \nu_j^{\frac1{3+\alpha}},
			\label{eq:2.10}
		\end{equation}
		where $\delta_{z_j}$ is the Dirac mass of mass 1 concentrated at
		$z_j$ and the subsequence $\{z_j\} \subset \mathbb R^3$.
	\end{lemma}

	\begin{lemma}[label=lemma:2.5, note=\cite{24}]
		Let $\{u_n\} \subset \mathcal D^{1,2}(\mathbb R^3)$ be a sequence in~\zcref[S]{lemma:2.4} and define
		\begin{gather*}
			\mu_\infty := \lim_{R\to\infty} \limsup_{n\to\infty}
			\int_{|x|\geq R} |\nabla u_n|^2 \d x,\\
			\nu_\infty := \lim_{R\to\infty} \limsup_{n\to\infty} \int_{|x| \geq R}
			(I_\alpha * |u|^{3+\alpha}) |u|^{3+\alpha} \d x.
		\end{gather*}
		Then it follows that
		\begin{gather}
			\limsup_{n\to\infty} \int_{\mathbb R^3} |\nabla u_n|^2 \d x
		= \int_{\mathbb R^3} \d \mu + \mu_\infty,\\
		\limsup_{n \to \infty} \int_{\mathbb R^3} (I_\alpha * |u|^{3+\alpha})
		|u|^{3+\alpha} \d x = \int_{\mathbb R^3} \d\nu + \nu_\infty
		\end{gather}
		and
		\begin{equation}
			S_{h,l}^2 \nu_\infty^{\frac{2}{3+\alpha}} \leq \mu_\infty
			\ab(\int_{\mathbb R^3} \d \mu + \mu_\infty).
		\end{equation}
	\end{lemma}

	\section{Proof of Theorem 1.1}
	\qquad For $u\in S_{r,b}$, by~\zcref[S]{lemma:2.2*} and the definition of $S_{h,l}$, we get
	\begin{align*}
		\begin{split}
			J(u)&=\dfrac{1}{2}\int_{\Omega_{r}}|\nabla u|^2\d x-\dfrac{1}{p}\int_{\Omega_{r}}|u|^{p} \d x+\dfrac{1}{4}\int_{\Omega_{r}}q(x) u^2\Phi(u)\d x-\dfrac{1}{2(3+\alpha)}\int_{\Omega_{r}}(I_{\alpha}*|u|^{3+\alpha})|u|^{3+\alpha}\d x\\
			&\geq \dfrac{1}{2}\int_{\Omega_{r}}|\nabla u|^2\d x-\dfrac{1}{p}\int_{\Omega_{r}}|u|^{p} \d x-\dfrac{1}{2(3+\alpha)}S_{h,l}^{-(3+\alpha)}\|\nabla u\|_{2}^{2(3+\alpha)}\\
			&\geq \dfrac{1}{2}\int_{\Omega_{r}}|\nabla u|^2\d x-\dfrac{1}{p}C_{N,p}\|u\|_{2}^{p-\frac{3(p-2)}{2}}\|\nabla u\|_{2}^{\frac{3(p-2)}{2}}-\dfrac{1}{2(3+\alpha)}S_{h,l}^{-(3+\alpha)}\|\nabla u\|_{2}^{2(3+\alpha)}\\
			&=\dfrac{1}{2}\|\nabla u\|_{2}^2-\dfrac{1}{p}C_{N,p}b^{p-\frac{3(p-2)}{2}}\|\nabla u\|_{2}^{\frac{3(p-2)}{2}}-\dfrac{1}{2(3+\alpha)}S_{h,l}^{-(3+\alpha)}\|\nabla u\|_{2}^{2(3+\alpha)}\\
			&:=h(\|\nabla u\|_{2}),
		\end{split}
	\end{align*}
	where
	\begin{align*}
		h(t):=\dfrac{1}{2}t^2-\dfrac{1}{p}C_{N,p}b^{p-\frac{3(p-2)}{2}}t^{\frac{3(p-2)}{2}}-\dfrac{1}{2(3+\alpha)}S_{h,l}^{-(3+\alpha)}t^{2(3+\alpha)}.
	\end{align*}
	
	By $0<\alpha<3, 2<p<\dfrac{8}{3}$, there exist $b^* > 0$ such that
	for all $b \in (0, b^*)$, $h(t)$ achieves its positive local maximum, i.e.,
	there exists $0 < R_1 < R_2 < +\infty$
	such that $h(t)>0$ for $t\in(R_{1},R_{2})$, and $h(t)<0$ for $t\in(0,R_{1})$ or $t\in(R_{2},+\infty)$. Let $\xi \in C^{\infty}(\mathbb R^{+},[0,1])$ satisfy
	\begin{align*}
		\xi(t)=\left\{\begin{array}{cl}
			1,&\text{if}\ \  t\leq R_{1},\\
			0,&\text{if}\ \  t\geq R_{2}.
		\end{array}\right.
	\end{align*}
	
	Now, we introduce the truncated functional
	\begin{align*}
		J^{T}(u):= {} & \dfrac{1}{2}\int_{\Omega_{r}}|\nabla u|^2\d x-\dfrac{1}{p}\int_{\Omega_{r}}|u|^{p}\d x+\dfrac{1}{4}\int_{\Omega_{r}}q(x) u^2\Phi(u)\d x\\
		&-\dfrac{\xi(\|\nabla u\|_{2})}{2(3+\alpha)}\int_{\Omega_{r}}(I_{\alpha}*|u|^{3+\alpha})|u|^{3+\alpha}\d x,
	\end{align*}
	where
	\begin{align*}
		\widetilde{h}(t):=\dfrac{1}{2}t^2-\dfrac{1}{p}C_{N,p}b^{p-\frac{3(p-2)}{2}}t^{\frac{3(p-2)}{2}}-\dfrac{\xi(t)}{2(3+\alpha)}S_{h,l}^{-(3+\alpha)}t^{2(3+\alpha)}.
	\end{align*}
	Then, by the definition of $\xi$ and since $b \in (0, b^*)$, we have $\widetilde{h}(t)<0$ for all $t\in(0,R_{1})$ and $\widetilde{h}(t)>0$ for all $t\in(R_{1},+\infty)$. Moreover, we will choose $R_{1}>0$ small enough such that
	\begin{equation}
		\dfrac{1}{2}r^2-\dfrac{1}{2(3+\alpha)}S_{h,l}^{-(3+\alpha)}r^{2(3+\alpha)}\geq 0,\ \ \ \text{for all }r\in[0,R_{1}]\ \ \text{and }R_{1}^2<S_{h,l}^{\frac{3+\alpha}{\alpha+2}}.
		\label{eq:3.2}
	\end{equation}
	\begin{lemma}[label=lemma:3.1]
		The functional $J^{T}$ has the following properties:
		\begin{enumerate}
			\item[(i)] $J^{T}\in C^{1}(H_{0}^{1}(\Omega_{r}),\mathbb R)$;
			\item[(ii)] $J^{T}$ is coercive and bounded from below on $S_{r,b}$. Moreover, if $J^{T}(u)\leq 0$, then $\|\nabla u\|_{2}\leq R_{1}$ and $J^{T}(u)=J(u)$;
			\item[(iii)] $J$ satisfies the $(PS)_{d}$ condition on $S_{r,b}$.
		\end{enumerate}
	\end{lemma}
	\begin{proof}
		(i) can be easily proved by standard argument.
		
		For (ii), if $\|\nabla u\|_{2}>R_{1}$ and considering $\|\nabla u\|_{2}\to+\infty$, we have
		\begin{align*}
			J^{T}(u)\geq\dfrac{1}{2}\|\nabla u\|_{2}^2-\dfrac{1}{p}C_{N,p}b^{p-\frac{3(p-2)}{2}}\|\nabla u\|_{2}^{\frac{3(p-2)}{2}}-\dfrac{\xi(\|\nabla u\|_{2})}{2(3+\alpha)}S_{h,l}^{-(3+\alpha)}\|\nabla u\|_{2}^{2(3+\alpha)}.
		\end{align*}
		Then, we can get $J^{T}$ is coercive. According to the definition of $\widetilde{h}(t)$, we derive $J^{T}(u)\geq \widetilde{h}(\|\nabla u\|_{2})=h(\|\nabla u\|_{2})$ when $\|\nabla u\|_{2}\in[0,R_{1}]$. Thus, $J^T$ is bounded from below on $S_{r,b}$.
		
		For (iii), define
		$$d:=\liminf_{u\in S_{r,b}}J^{T}(u).$$
		For any $u\in S_{1,b}$, the scaling
		\begin{equation}
			u_{t}(x): = t^{\frac{3}{2}}u(tx)\ \ \ \ \text{for}\ \ \ \ t\in\mathbb R,\ \ \  x\in\mathbb R^3
			\label{eq:3.2*}
		\end{equation}
		preserves the $L^2$-norm, i.e.,
		\begin{align*}
			\|u_{t}(x)\|_{2}=\int_{\Omega_{r}}|t^{\frac{3}{2}}u(tx)|^2\d x=t^3\int_{\Omega_{\frac{1}{t}}}|u(tx)|^2\d x=\int_{\Omega}|u(y)|^2dy=b^2.
		\end{align*}
		
		By using H\"{o}lder inequality, and $H_{0}^{1}(\Omega_{r})\hookrightarrow L^{6}(\Omega_{r})\hookrightarrow L^2(\Omega_{r})$, we have
		\begin{align*}
			\int_{\Omega_{r}}|\Phi(u_{t})|^2\d x&\leq \left(\int_{\Omega_{r}}|\Phi(u_{t})|^{6}\d x\right)^{\frac{1}{3}}|\Omega_{r}|^{\frac{2}{3}} = \left(\int_{\Omega_{r}}|\Phi(u_{t})|^{6}\d x\right)^{\frac{1}{3}}(r^3|\Omega|)^{\frac{2}{3}}\\
			&= r^2\left(\int_{\Omega_{r}}|\Phi(u_{t})|^{6}\d x\right)^{\frac{1}{3}}|\Omega|^{\frac{2}{3}} \\
			&\leq r^2\left(S^{-3}\|\nabla \Phi(u_{t})\|_{2}^{6}\right)^{\frac{1}{3}}|\Omega|^{\frac{2}{3}} = r^2S^{-1}\|\nabla \Phi(u_{t})\|_{2}^2|\Omega|^{\frac{2}{3}},
		\end{align*}
		where $S^{\frac{1}{2}}$ is the constant in Sobolev embedding $D^{1,2}(\mathbb R^3)\hookrightarrow L^{6}(\mathbb R^3)$.
		
		Since
		\begin{align*}
			\int_{\Omega_{r}}|\nabla\Phi(u_{t})|^2\d x\leq \int_{\Omega_{r}}(|\nabla\Phi(u_{t})|^2+|\Delta\Phi(u_{t})|^2)\d x=\int_{\Omega_{r}}q(x) u_{t}^2\Phi(u_{t})\d x,
		\end{align*}
		we get
		\begin{equation}
			\begin{aligned}
				\int_{\Omega_{r}}q(x)u_{t}^2\Phi(u_{t})\d x&\leq \|q(x)\|_\infty\|u_{t}^2\|_{2}\|\Phi(u_{t})\|_{2}\leq \|q(x)\|_\infty\|u_{t}\|_{4}^2(rS^{-\frac{1}{2}}|\Omega|^{\frac{1}{3}})\|\nabla\Phi(u_{t})\|_{2}\\
				&\leq \|q(x)\|_\infty S^{-\frac{1}{2}}|\Omega|^{\frac{1}{3}}r\|u_{t}\|_{4}^2\left(\int_{\Omega_{r}}q(x)u_{t}^2\Phi(u_{t})\d x\right)^{\frac{1}{2}}.
			\end{aligned}
			\label{eq:*}
		\end{equation}
		Note that
		\begin{equation}
			\begin{aligned}
				\|u_{t}\|_{4}^2 & =\left(\int_{\Omega_{r}}|u_{t}|^{4}\d x\right)^{\frac{1}{2}}=\left(\int_{\Omega_{r}}|t^{\frac{3}{2}}u(tx)|^{4}\d x\right)^{\frac{1}{2}}\\
				& =t^3\left(\int_{\Omega_{r}}|u(tx)|^{4}\d x\right)^{\frac{1}{2}}=t^{\frac{3}{2}}\left(\int_{\Omega}|u(x)|^{4}\d x\right)^{\frac{1}{2}}.
			\end{aligned}
			\label{eq:**}
		\end{equation}
		Let $r = \frac1t$, and from~\eqref{eq:*} and \eqref{eq:**} we obtain
		\[
			\int_{\Omega_{\frac1t}} q(x) u_t^2 \Phi(u_t) \d x \leq
			\ab(\|q(x)\|_\infty S^{-\frac12}|\Omega|^{\frac13}) t^{\frac12}
			\ab(\int_\Omega |u|^4 \d x)^{\frac12} \ab(\int_{\Omega_{\frac1t}} q(x) u_t^2 \Phi(u_t) \d x)^{\frac12}.
		\]
		Denote $\tilde c := \|q(x)\|_\infty S^{-\frac12}|\Omega|^{\frac13}$, it
		follows that
		\begin{equation}
			\int_{\Omega_{\frac1t}} q(x) u_t^2 \Phi(u_t) \d x \leq \tilde c^2 t\int_{\Omega}|u|^4 \d x.
			\label{eq:3.6}
		\end{equation}
	
		On the other hand, take $u_0 \in S_{1,b}$, substitute~\eqref{eq:3.2*} into $J^T$,
		and consider $t$ is relatively small, we derive from~\eqref{eq:3.6} that
		\begin{align*}
			J^T(u_t(x)) & = \frac12 \int_{\Omega_{\frac1t}} |\nabla u_t|^2 \d x
		- \dfrac{1}{p}\int_{\Omega_{\frac1t}} |u_t|^{p} \d x + \frac14 \int_{\Omega_{\frac1t}} q(x)u_t^2 \Phi(u_t) \d x\\
		& \quad - \frac{t^{3 + 3\alpha}}{2(3 + \alpha)} \int_\Omega (I_\alpha * |u_0|^{3+\alpha}) |u_0|^{3+\alpha} \d x\\
		& \leq \frac{t^2}{2} \int_\Omega |\nabla u_0|^2 \d x - \dfrac{1}{p}t^{\frac{3p}{2}-3}
		\int_\Omega |u_0|^{p} \d x + \frac{\tilde c^2}{4}t \int_\Omega |u_0|^4 \d x.
		\end{align*}
	By Ekeland's variational principle, there exists a sequence $\{u_n\} \subset S_{r,b}$, which is a Palais-Smale (PS) sequence for $J^T(u)$. Since $p\in\ab(2,\frac38)$,
	then there exists $t_0 > 0$ small enough such that $J^T(u_{t_0}(x)) < 0$.

	Now, we fix $r = \frac1{t_0}$ (hence $\Omega_r = \Omega _{\frac1{t_0}}$).
	Thus, we have determined the sizes of $\Omega_r$.
	Moreover, we know that for sufficiently large $n$, $J^T(u_n) < 0$, so $\|\nabla u_n\|_2 \leq R_1$ from (ii), and $J(u_n) = J^T(u_n)$. Thus $\{u_n\}$ is a $(PS)_d$ sequence of $J$, i.e.,
	\[
		J(u_n) \to d < 0, \quad \|J_{s,b}'(u_n)\| \to 0, \quad \text{as $n \to \infty$}.
	\]
Since we have already known that $\Omega_r = \Omega_{\frac1{t_0}}$, for convenience, we will use $\Omega_r$ to replace $\Omega_{\frac1{t_0}}$ in the following context.

	Furthermore, we know that $\{u_n\}$ is a bounded sequence in $H_0^1(\Omega_r)$,
	so there exists a subsequence (still denoted $\{u_n\}$) and $\bar u \in H_0^1(\Omega_r)$ such that
	\begin{equation}
		u_n \rightharpoonup \bar u \quad \text{in $H_0^1(\Omega_r)$},
		\quad
		u_n \rightharpoonup \bar u \quad \text{in $L^t(\Omega_r)$},
		\quad \text{for $2 \leq t < 6$}.
		\label{eq:3.7}
	\end{equation}
	Thus, from $p \in (2, \frac38)$, we obtain
	\[
		\lim_{n\to\infty} \int_{\Omega_r} |u_n|^{p} \d x
	= \int_{\Omega_r} |\bar u|^{p} \d x.
	\]
	
	Next, we prove $\bar u \neq 0$. Suppose by contradiction that $\bar u = 0$, then
	\[
		\lim_{n\to\infty} \int_{\Omega_r} |u_n|^{p} \d x
		= \int_{\Omega_r} |\bar u|^{p} \d x=0.
	\]
	Further, we have
	\begin{align*}
		0 > d & = \lim_{n\to\infty} J^T(u_n) = \lim_{n\to\infty} J(u_n)\\
		& \geq \lim_{n\to\infty} \ab(\frac12 \int_{\Omega_r} |\nabla u_n|^2 \d x
		- \dfrac{1}{p} \int_{\Omega_r} |u_n|^{p} \d x
		- \frac1{2(3 + \alpha)} S_{h,l}^{-(3+\alpha)}\|\nabla u_n\|_2^{2(3+\alpha)})\\
		& \geq -\dfrac{1}{p} \lim_{n\to\infty} \int_{\Omega_r} |u_n|^{p} \d x = 0,
	\end{align*}
	which is a contradiction. Thus, $\bar u \neq 0$.
	
	By Proposition 5.2 in~\cite{29}, for any $\varphi \in H_0^1(\Omega_r)$, there exists $\{\lambda_n\} \subset \mathbb R$ such that
	\begin{equation}
		\begin{aligned}
			\int_{\Omega_r} \nabla u_n \nabla \varphi \d x
		- \int_{\Omega_r} |u_n|^{p-2}u_n \varphi \d x
		+ \int_{\Omega_r} q(x) \Phi(u_n) u_n \varphi\d x\\
		- \int_{\Omega_r} (I_\alpha * |u_n|^{3+\alpha})|u_n|^{1+\alpha} u_n \varphi \d x
		- \lambda_n \int_{\Omega_r} u_n \varphi \d x = 0.
		\end{aligned}
		\label{eq:3.8*}
	\end{equation}
	Take $\varphi = u_n$, we get
	\begin{equation}
		\begin{aligned}
			\int_{\Omega_r} |\nabla u_n|^2 \d x
		- \int_{\Omega_r} |u_n|^{p} \d x
		+ \int_{\Omega_r} q(x) u_n^2 \Phi(u_n) \d x\\
		- \int_{\Omega_r} (I_\alpha * |u_n|^{3+\alpha}) |u_n|^{3+\alpha} \d x
	= \lambda_n \int_{\Omega_r} u_n^2 \d x = \lambda_n b^2.
		\end{aligned}
		\label{eq:3.9}
	\end{equation}
	Moreover,
	\begin{equation}
		\int_{\Omega_r} |u_n|^{p} \d x \leq
		\widehat c \|\nabla u_n\|_2^{p},
	\end{equation}
	where $\widehat c > 0$ is a constant.
	By applying~\eqref{eq:2.6*} and \eqref{eq:2.6}, we obtain
	\begin{equation}
		\ab(\int_{\Omega_r} q(x) u_n^2 \Phi(u_n) \d x)^{\frac12}
	= \|\Phi(u_n)\|_{\mathbb H} \leq C_2 \int_{\Omega_r} |\nabla u_n|^2\d x.
	\end{equation}
	
	In addition, by the definition of $S_{h,l}$, we obtain
	\begin{equation}
		\int_{\Omega_r} \ab(I_\alpha * |u_n|^{3+\alpha}) |u_n|^{3+\alpha} \d x
		\leq S_{h,l}^{3+\alpha} \ab(\int_{\Omega_r} |\nabla u_n|^2 \d x)^{3+\alpha}.
		\label{eq:3.12*}
	\end{equation}
	Combining with~\eqref{eq:3.9}-\eqref{eq:3.12} and the boundedness of $\{u_n\}$ in $H_0^1 (\Omega_r)$,
	it is easy to see that $\{\lambda_n\}$ is bounded. Up to a subsequence, we may
	assume $\lambda_n \to \lambda \in \mathbb R$.
	By~\eqref{eq:3.7}, it follows that
	\begin{equation}
		-\Delta \bar u + q(x)\Phi(\bar u) \bar u = \lambda \bar u + (I_\alpha * |\bar u|^{3+\alpha}) |\bar u|^{1+\alpha} \bar u + |\bar u|^{p-2}\bar u.
		\label{eq:3.12}
	\end{equation}
	Indeed, for any $\varphi \in H_0^1(\Omega_r)$, form $u_n \rightharpoonup \bar u$ in $H_0^1(\Omega_r)$ and $u_n \to \bar u$ in $L^p(\Omega_r)$
	for $p \in [2, 6)$, it follows that
	\begin{gather}
		\int_{\Omega_r} \nabla u_n \nabla\varphi \d x \to
		\int_{\Omega_r} \nabla \bar u \nabla \varphi \d x \quad \text{as $n \to \infty$},\label{eq:3.13}\\
		\int_{\Omega_r} |u_n|^{p-2}u_{n}\varphi \d x \to
		\int_{\Omega_r} |\bar u|^{p-2}\bar u \varphi \d x \quad \text{as $n \to \infty$},\\
		\lambda_n \int_{\Omega_r} u_n\varphi \d x \to
		\lambda \int_{\Omega_r} \bar u \varphi \d x \quad \text{as $n \to \infty$}.
	\end{gather}
	By~\zcref[S]{lemma:2.3}, we obtain
	\begin{equation}
		\int_{\Omega_r} q(x) \Phi(u_n) u_n \varphi \d x \to
		\int_{\Omega_r} q(x) \Phi(\bar u) \bar u \varphi \d x.
	\end{equation}
	By Sobolev embedding, we know that $u_n \rightharpoonup u$ weakly in $L^6(\Omega_r)$,
	$u_n \rightharpoonup u$ a.e. in $\Omega_r$. Then
	\[
		|u_n|^{3+\alpha} \rightharpoonup |\bar u|^{3+\alpha} \quad
		\text{weakly in $L^\frac{6}{3+\alpha}(\Omega_r)$}
	\]
	as $n \to \infty$. In view of~\eqref{eq:2.9*}, we know that
	the sequence $I_\alpha * |u_n|^{3+\alpha}$ is bounded in $L^{\frac{6}{3-\alpha}} (\Omega_r)$. So,
	\[
		I_\alpha * |u_n|^{3+\alpha} \rightharpoonup I_\alpha * |\bar u|^{3+\alpha}
		\quad \text{weakly in $L^{\frac{6}{3-\alpha}}(\Omega_r)$}
	\]
	as $n \to \infty$. Combing this and the fact that
	\[
		|u_n|^{3+\alpha} u_n \rightharpoonup |\bar u|^{3+\alpha}
		\quad
		\text{weakly in $L^{\frac{6}{2+\alpha}}(\Omega_r)$}
	\]
	as $n \to \infty$, we derive at
	\[
		(I_\alpha * |u_n|^{3+\alpha})|u_n|^{1+\alpha} u_n
		\rightharpoonup (I_\alpha * |\bar u|^{3+\alpha}) |\bar u|^{1+\alpha} u
		\quad \text{weakly in $L^{\frac65}(\Omega_r)$}
	\]
	as $n \to \infty$. Therefore, we have, for any $\varphi \in H_0^1(\Omega_r)$,
	\begin{equation}
		\int_{\Omega_r} (I_\alpha * |u_n|^{3+\alpha}) |u_n|^{1+\alpha}
		u_n\varphi \d x \to
		\int_{\Omega_r} (I_\alpha * |\bar u|^{3+\alpha}) |\bar u|^{1+\alpha} u\varphi \d x.
		\label{eq:3.17}
	\end{equation}
	
	So from~\eqref{eq:3.13}-\eqref{eq:3.17}, we get~\eqref{eq:3.12}.
	Next, we prove that
	\begin{gather*}
		\int_{\Omega_r} |\nabla u_n|^2 \d x \to \int_{\Omega_r} |\nabla \bar u| \d x,\\
		\int_{\Omega_r} (I_\alpha * |u_n|^{3+\alpha})|u_n|^{3+\alpha} \d x
		\to
		\int_{\Omega_r} (I_\alpha * |\bar u|^{3+\alpha}) |\bar u|^{3+\alpha} \d x.
	\end{gather*}
	We regard the sequence $u_n$ in $H_0^1(\Omega_r)$
	as elements of $H^1(\mathbb R^3)$ by extending them via zero outside $\Omega_r$, then we apply the compactness principle.
	For sufficiently large $n$, $\|\nabla u_n\| \leq R_1$; by~\zcref[S]{lemma:2.4} and~\zcref[S]{lemma:2.5}, there exist positive measures $\mu$, $\nu$,
	and $\mu$, $\nu$ are zero outside $\Omega_r$, such that
	\[
		|\nabla u_n|^2 \rightharpoonup \mu, \quad
		(I_\alpha * |u_n|^{3+\alpha})|u_n|^{3+\alpha} \rightharpoonup \nu
	\]
	as $n \to \infty$.
	
	Setting the functional $\Psi(v): H_0^1(\Omega_r) \to \mathbb H$ given by $\Psi(v) = \frac12\int_{\mathbb R^3}v^2\d x$.
	Therefore, we can see that $S_{r,b} = \Psi^{-1}\ab(\ab\{\frac{b^2}{2}\})$.

 First, we show that $I$ is empty.
 By the conservation of the mass, $I$ is not infinite. If $I$ is nonempty and finite, then
$\mu_j$ and $\nu_j$ are positive measures derived from the~\zcref[S]{lemma:2.4}.

	Let $\varphi \in C_0^\infty (\mathbb R^3)$ be a cut-off function such that $\varphi \in [0,1]$, $\varphi \equiv 1$ in $B_{\frac12}(0)$, and $\varphi \equiv 0$ in $\mathbb R^3\backslash B_1(0)$. For any $\rho > 0$, set
	\[
		\varphi_\rho (x) := \varphi\ab(\frac{x - x_j}{\rho})
	= \begin{cases}
		1, & \text{if $|x - x_j| \leq \frac\rho2$},\\
		0, & \text{if $|x - x_j| \geq \rho$}.
	\end{cases}
	\]
	
	Since $\{u_n\}$ is bounded in $H_0^1(\Omega_r)$, $\{\varphi_pu_n\}$ is also bounded in $H_0^1(\Omega_r)$. Thus
	\begin{equation}
		\begin{aligned}
			o_n(1) & = \braket<J'(u_n) - \lambda_n \Psi'(u_n), \varphi_\rho u_n>\\
	& = \int_{\mathbb R^3} \varphi_\rho |\nabla u_n|^2 \d x
		+ \int_{\mathbb R^3} u_n \nabla u_n \nabla \varphi_\rho \d x
		- \int_{\mathbb R^3} |u_n|^{p} \varphi_\rho u_n \d x\\
	& \quad+ \int_{\mathbb R^3} q(x) u_n^2 \Phi(u_n) \varphi_\rho \d x
		- \lambda_n \int_{\mathbb R^3} u_n^2 \varphi_\rho \d x
		- \int_{\mathbb R^3} (I_a*|u_n|^{3+\alpha}) |u_n|^{3+\alpha}
			\varphi_\rho \d x.
		\end{aligned}
		\label{eq:3.18}
	\end{equation}
	First,
	we regard $u_n$ and $\bar u$ as functions in $H^1(\mathbb R^3)$ (with zero extension outside $\Omega_r$).
	Combing the absolute continuity of the Lebesgue integral and $u_n \rightharpoonup \bar u$ in $H_0^1(\Omega_r)$, $u_n\rightharpoonup \bar u$ in $L^p(\Omega_r)$ for $p \in [2,6)$, we obtain from~\zcref[S]{lemma:2.4} that
	\begin{gather}
		\lim_{\rho\to0} \lim_{n\to0} \int_{\mathbb R^3} u_n \nabla u_n \nabla \varphi_\rho \d x = 0,\\
		\lim_{\rho \to 0} \lim_{n\to\infty} \int_{\mathbb R^3}
		|u_n|^{p} \varphi_\rho = \lim_{\rho\to 0} \int_{\mathbb R^3} |\bar u|^{p} \varphi_\rho \d x = 0,\\
		\lim_{\rho\to0} \lim_{n\to\infty} \int_{\mathbb R^3} q(x) u_n^2 \Phi(u_n) \varphi_\rho
	= \lim_{\rho\to0} \int_{\mathbb R^3} q(x) \bar u^2 \Phi(\bar u) \varphi_\rho = 0,
	\end{gather}
	and
	\begin{equation}
		\lim_{\rho\to0} \lim_{n\to\infty} \int_{\mathbb R^3} \varphi_\rho u_n^2 \d x = 0.
		\label{eq:3.21}
	\end{equation}
	
	Then, by \eqref{eq:3.18}-\eqref{eq:3.21}, \zcref[S]{lemma:2.4} and the definition of $\varphi_\rho$, it follows that
		\begin{align*}
			& \lim_{\rho \to 0} \ab\{\int_{\mathbb R^3} (I_\alpha * |u|^{3+\alpha})|u|^{3+\alpha}\varphi_\rho + \sum_{j\in I} \nu_j \delta_{x_j} \varphi_\rho \d x\}\\
		= & \lim_{\rho \to 0} \int_{\mathbb R^3} \varphi_\rho \d \nu
		= \lim_{\rho\to0} \lim_{n\to\infty} \int_{\mathbb R^3} (I_a * |u_n|^{3+\alpha})
			|u_n|^{3+\alpha} \varphi_\rho \d x\\
		= & \lim_{\rho \to 0} \lim_{n\to\infty} \int_{\mathbb R^3} |\nabla u_n|^2
		\varphi_\rho \d x = \lim_{\rho\to0} \int_{\mathbb R^3} \varphi_\rho \d \mu\\
		\geq & \lim_{\rho\to 0} \ab\{\int_{\mathbb R^3} |\nabla u|^2 \varphi_\rho \d x + \sum_{j\in I} \mu_j \delta_{x_j} \varphi_\rho \d x\} = \mu_j.
		\end{align*}
	So we can deduce that $\nu_j \geq \mu_j$.
	Combining with~\eqref{eq:2.10},
	we can derive that $\mu_j \geq S_{h,l} \nu_j^{\frac1{3+\alpha}} \geq S_{h,l} \mu_j^{\frac1{3+\alpha}}$.
	Consequently, one has $\mu_j \geq S_{h,l} \mu_j^{\frac1{3+\alpha}}$, thus
	\[
		R_1^2 \geq \lim_{\rho\to 0} \lim_{n\to\infty}
		\|\nabla u_n\|_2^2 \geq \lim_{\rho\to0} \lim_{n\to\infty}
		\int_{\mathbb R^3} |\nabla u_n|^2 \varphi_\rho \d x
	= \lim_{\rho \to0} \int_{\mathbb R^3} \varphi_\rho \d \mu \geq 
		\mu_j \geq S_{h,l}^{\frac{3+\alpha}{\alpha+2}},
	\]
	which contradicts~\eqref{eq:3.2}, then $I$ is an empty set.
	
	Considering the fact that $\mu$ and $\nu$ vanish outside $\Omega_r$, we have $\nu_\infty = M_\infty = 0$.
	By the~\zcref[S]{lemma:2.5}, we further obtain
	\[
		\int_{\mathbb R^3}(I_\alpha * |u_n|^{3+\alpha}) |u_n|^{3+\alpha} \d x \to
		\int_{\mathbb R^3} (I_a*|\bar u|^{3+\alpha}) |\bar u|^{3+\alpha} \d x.
	\]
	Since $u_n \rightharpoonup \bar u$, taking $\varphi = \bar u$ in equation~\eqref{eq:3.8*}, then subtracting~\eqref{eq:3.9}, we obtain
	\begin{align*}
		&\bigg(\int_{\mathbb R^3} |\nabla u_n|^2 \d x + \int_{\mathbb R^3} q(x) \Phi(u_n) u_n^2 \d x - \int_{\mathbb R^3} |u_n|^{p} \d x - \lambda_n \int_{\mathbb R^3} u_n^2 \d x\\
		&- \int_{\mathbb R^3} (I_a * |u_n|^{3+\alpha}) |u_n|^{3+\alpha} \d x\bigg)\\
		&- \bigg(\int_{\mathbb R^3} |\nabla \bar u|^2 \d x +
		\int_{\mathbb R^3} q(x) \Phi(\bar u) \bar u^2 \d x -
		\int_{\mathbb R^3} |\bar u|^{p}\d x\\
		&- \lambda \int_{\mathbb R^3} \bar u^2\d x
		- \int_{\mathbb R^3} (I_a * |\bar u|^{3+\alpha})|\bar u|^{3+\alpha}  \d x\bigg) = o_n(1).
	\end{align*}
	Since $u_n \rightharpoonup \bar u$ in $H_0^1(\Omega_r)$, $u_n \to \bar u$
	in $L^p(\Omega_r)$, where $p \in [2,6)$; and (i) of~\zcref[S]{lemma:2.3}, we can obtain
	\[
		\lim_{n\to\infty} \int_{\Omega_r} |\nabla u_n|^2 \d x =
		\int_{\Omega_r} |\nabla \bar u|^2 \d x,
	\]
  i.e., $u_n \to \bar u \quad \text{in $H_0^1(\Omega_r)$}$.
	\end{proof}
	
	For any $\varepsilon > 0$, we define the set
	\[
		(J^T)^{-\varepsilon} := \{u \in S_{r,b}: J^T(u) \leq -\varepsilon\}
		\subset H_0^1(\Omega_r).
	\]
	Since $J^T(u)$ is coercive and even on $H_0^1(\Omega_r)$, we know $(J^T)^{-\varepsilon}$ is closed and symmetric. To prove~\zcref[S]{thm:1.1},
	we present the following two lemmas. The proof of~\zcref[S]{lemma:3.2} can be completed by analogy with Lemma 3.2 in~\cite{alves2021multiplicity}.
	\begin{lemma}[label=lemma:3.2]
		Given $k \in \mathbb N$, there exist $\varepsilon_k = \varepsilon(k) > 0$
		such that $\gamma((J^T)^{-\varepsilon}) \geq k$ for all $0 < \varepsilon \leq \varepsilon_k$.
	\end{lemma}
		For any $\varepsilon_k > 0$, we define the set
		\[
			\Sigma_k := \{D \subset S_{r,b}:\ \text{$D$ is closed and symmetric, $\gamma(D) \geq k$}\}.
		\]
		By part (ii) of~\zcref[S]{lemma:3.1}, we define the minimax value
		\[
			d_k := \inf_{D\in\Sigma_k} \sup_{u\in D} J^T (u) > -\infty
		\]
		for all $k \in \mathbb  N$.
		In addition, we define the set
		\[
			K_d := \{u \in S_{r,b}: (J^T)'(u) = 0,\ J^T(u) = d\}.
		\]
	Then, we can prove the following result.
	\begin{lemma}[label=lemma:3.3]
		If $d = d_k = d_{k+1} = \ldots = d_{k+l}$, then
		$\gamma(K_d) \geq l + 1$. In particular, $J^T$ has at least $l + 1$
		non-trivial critical points.
		\begin{proof}
			For $\varepsilon > 0$, it is easy to see $(J^T)^{-\varepsilon} \subset \Sigma$.
			For any fixed $k \in \mathbb N$, by~\zcref[S]{lemma:3.2}, we have $(J^T)^{-\varepsilon} \in \Sigma_k$. Moreover,
			$d_k \leq \sup\limits_{u\in(J^T)^{-\varepsilon_k}} J^T(u) = -\varepsilon_k < 0$.

			If $0 > d = d_k = d_{k+1} = \ldots = d_{k+l}$, then part (iii) of the~\zcref[S]{lemma:3.1} implies $J^T$ satisfies the Palais-Smale $(PS)_d$ condition on $S_{r,b}$ for $d<0$.
			Therefore, we deduce that $K_d$ is compact.
			By Theorem 2.1 in~\cite{28}, there are at least $l + 1$ non-trivial critical points for $J^T\big|_{S_{r,b}}$.
		\end{proof}
	\end{lemma}
	\noindent
	\textbf{Proof of~\zcref[S]{thm:1.1}.}
	By part (ii) of~\zcref[S]{lemma:3.1}, the critical points of $J$ are critical points of $J^T$. So from~\zcref[S]{lemma:3.3}, the proof of~\zcref[S]{thm:1.1} completes.
	\hfill\qedsymbol

	\section{Declarations}

	\paragraph{Acknowledgements}
	Li Wang was supported by National Natural Science Foundation of China (Grant Nos. 12161038, 12301584), Jiangxi Provincial Natural Science Foundation (Grant No. 20232BAB201009), Science and Technology Project of Jiangxi Provincial Department of Education (Grant No. GJJ2400901).

	\paragraph{Author Contributions}
	All authors have accepted reponsibility for the entire content of this version of the manuscript and consented to its submission to the journal, reviewed all the results and approved the final version of the manuscript.
	Li Chen proposed research ideas and writing original draft preparation.
	Li Wang provided this idea for the study and led the implementation review and revision of the manuscript. All authors have read and agreed to the published version of the manuscript.

	\paragraph{Conflict of interest}
	The authors declare that they have no conflict of interest.

\end{document}